\newcount\Comments  % 0 suppresses notes to selves in text
\documentclass[letterpaper,10pt]{scrartcl}

\usepackage[utf8]{inputenc}
\usepackage[english]{babel}
\usepackage[automark]{scrlayer-scrpage}
\usepackage{amsmath,amstext,amsbsy,amsopn,amscd,amsxtra,upref,amssymb}
\usepackage{amsfonts,bm}
\usepackage{amsthm}
\usepackage{color}
\usepackage{graphicx}
\usepackage{sidecap}
\usepackage{multirow}
\usepackage{booktabs}
\usepackage{bm}
\usepackage{fullpage}

\usepackage{tabularx}
\usepackage[font=small]{caption}
\usepackage{subfig}

\usepackage{longtable}
\usepackage{rotating}

\usepackage{algorithm}
\usepackage{algpseudocode}
\extrafloats{100}

\usepackage{lineno,hyperref}

\graphicspath{{figures/}}
\DeclareGraphicsExtensions{.pdf,.eps,.png,.jpg,.jpeg}

\newcommand{\Ccal}{\mathcal{C}}

\newcommand{\Ncal}{\mathcal{N}}

\newcommand{\bfeta}{\boldsymbol \eta}

\newcommand{\sL}{L^{(s)}}

\newcommand{\bfepsilon}{\boldsymbol{\epsilon}}

\newcommand{\nh}{N} % FOM dimension
\newcommand{\nr}{n} % pod dimension
\newcommand{\tH}{\widetilde{H}}
\newcommand{\tL}{\widetilde{L}}
\newcommand{\tsL}{\widetilde{L}^{(s)}}
\newcommand{\hatA}{\widehat{A}}
\newcommand{\hatE}{\widehat{E}}
\newcommand{\hatx}{\hat{x}}
\newcommand{\haty}{\hat{y}}
\newcommand{\hatB}{\widehat{B}}
\newcommand{\hatC}{\widehat{C}}
\newcommand{\hatH}{\widehat{H}}
\newcommand{\tC}{\widetilde{C}}
\newcommand{\tB}{\widetilde{B}}
\newcommand{\tE}{\widetilde{E}}
\newcommand{\tA}{\widetilde{A}}
\newcommand{\tx}{\tilde{x}}
\newcommand{\ty}{\tilde{y}}
\newcommand{\bfiota}{\boldsymbol \iota}

\newcommand{\Real}{\mathfrak{R}}
\newcommand{\Imag}{\mathfrak{I}}

\newcommand{\kibitz}[2]{\ifnum\Comments=1\textcolor{#1}{#2}\fi}

\newenvironment{keywords}%
   {\begin{trivlist}\item[]{\bfseries\sffamily Keywords:}\ }% oder "Keywords:"
   {\end{trivlist}}

\ihead[]{}
\ohead[]{}
\ifoot[]{}
\ofoot[]{}

\newtheorem{lemma}{Lemma}
\newtheorem{theorem}{Theorem}
\newtheorem{corollary}{Corollary}
\theoremstyle{definition}

\newtheorem{proposition}{Proposition}

\ifpdf
\hypersetup{
  pdftitle={Learning low-dimensional dynamical-system models from noisy frequency-response data with Loewner rational interpolation},
  pdfauthor={Zlatko Drma\v{c} and Benjamin Peherstorfer}
}
\fi

\title{Learning low-dimensional dynamical-system models from noisy frequency-response data with Loewner rational interpolation}

\author{Zlatko Drma\v{c}\thanks{Faculty of Science, Department of Mathematics, University of Zagreb, Bijeni\v{c}ka 30, 10000 Zagreb, Croatia} \and Benjamin Peherstorfer\thanks{Courant Institute of Mathematical Sciences, New York University, New York, NY 10012, USA}}
\date{November 2020}

\begin{document}

\maketitle

\begin{center}
\emph{Dedicated to Athanasios C Antoulas on the occasion of his 70th birthday.}
\end{center}

\begin{abstract}
	Loewner rational interpolation
	provides a versatile tool to learn low-dimensional dynamical-system models from frequency-response measurements. This work investigates the robustness of the Loewner approach to noise. The key finding is that if the measurements are polluted with Gaussian noise, then the error due to noise grows at most linearly with the standard deviation with high probability under certain conditions. The analysis gives insights into making the Loewner approach robust against noise via linear transformations and judicious selections of measurements. Numerical results demonstrate the linear growth of the error on benchmark examples.
\end{abstract}


\section{Introduction}
Learning dynamical-system models from measurements is a widely studied task in science, engineering, and machine learning; see, e.g., system identification originating from the systems \& control community \cite{NewFIR,FIRSystem,Mendel,LjungBook,Viberg19951835,KramerERZ,Qin20061502,Reynders2012,772353,VectorFitting,doi:10.1137/15M1010774}, sparsity-promoting methods \cite{Schaeffer6634,doi:10.1137/18M116798X,Brunton12042016,Rudye1602614}, dynamic mode decomposition \cite{SchmidDMD,FLM:7843190,FLM:6837872,Tu2014391,NathanBook, DDMD-SISC-2018}, and operator inference \cite{Peherstorfer16DataDriven,pehersto15dynamic,doi:10.2514/6.2019-3707,KPW16ControlAdaptROM,Peherstorfer16AdaptROM,2019arXiv190803620S,PReProj}. Antoulas and collaborators introduced the Loewner approach \cite{ANTOULAS01011986,5356286,Mayo2007634,BeaG12} that constructs models directly from frequency-response measurements, without requiring computationally expensive training phases and without solving potentially non-convex optimization problems. This work investigates the robustness of the Loewner approach to noise in the frequency-response measurements. The main finding of this work is that under certain conditions and with high probability the error introduced by noise into Loewner models grows at most linearly with the standard deviation of the noise.

The Loewner approach \cite{ANTOULAS01011986,5356286,Mayo2007634,BeaG12} derives dynamical-system models from frequency-response data, i.e., from values of the transfer function of the high-dimensional dynamical system of interest. A series of works have extended the Loewner approach from linear time-invariant systems to bilinear systems \cite{doi:10.1137/15M1041432}, quadratic-bilinear systems \cite{doi:10.1002/nla.2200}, parametrized systems \cite{doi:10.1137/130914619}, time-delay systems \cite{SCHULZE2016125}, and structured systems \cite{SCHULZE2018250}. Learning Loewner models from time-domain data, instead of frequency-response measurements, is discussed in  \cite{LoewnerNonInt,PSW16TLoewner}. Learning Loewner models from noisy data has received relatively little attention. The work \cite{Lefteriu2010} provides a numerical investigation of the effect of noise on the accuracy of Loewner models. The work \cite{2018arXiv180300043H} discusses the rank of Hankel matrices if measurements are polluted by additive Gaussian noise. Numerical experiments in the thesis \cite[Section~2.1]{CosminsThesis} demonstrate that the selection of frequencies at which to obtain measurements and how to partition the measurements has a significant influence on the robustness against perturbations in the data such as noise. The work \cite{8550216} applies the Loewner approach to control tasks where only noisy measurements are available. The authors of \cite{BEATTIE20122916} discuss the robustness of interpolatory model reduction against perturbations in evaluations of the transfer function of the high-dimensional system; however, the perturbations are considered deterministic and stem from, e.g., numerical approximations via iterative methods. During the writing of this manuscript, the authors became aware of the work \cite{embree2019pseudospectra} that studies the sensitivity of Loewner interpolation to perturbations in a deterministic fashion via pseudospectra of the Loewner matrix pencils.

In this work, the robustness of the Loewner approach to Gaussian noise is considered. Rather than a deterministic analysis, the contribution of this work is an analysis that bounds the error introduced by noise in a probabilistic sense. In particular, the analysis shows that the error grows at most linearly in the standard deviation of the noise under certain conditions and with high probability. A relative noise model is considered, which seems realistic because measurement errors typically are relative to the value of the measured quantities. The conditions under which the proposed error bounds hold give insights into selecting the frequencies at which to measure and partitioning the data to reduce the effect of noise. The linear growth of the error with respect to the standard deviation of the noise is observed in numerical examples.

\section{Learning low-dimensional dynamical-system models with Loewner rational interpolation}
This section recapitulates model reduction with the Loewner approach; see, e.g., \cite{AntBG10,MORSurveySIREV,PWG17MultiSurvey} for general introductions to model reduction and related concepts.

\subsection{Linear time-invariant dynamical systems}
Consider the linear time-invariant system
\begin{equation}
E\dot{x}(t) = Ax(t) + Bu(t)\,,\qquad y(t) = C x(t)
\label{eq:Prelim:FOM}
\end{equation}
of order $\nh \in \mathbb{N}$ with system matrices $E, A \in \mathbb{R}^{\nh \times \nh}, B \in \mathbb{R}^{\nh \times 1}, $ and $C \in \mathbb{R}^{1 \times \nh}$. The state at time $t$ is $x(t) \in \mathbb{R}^{\nh}$ and the output at time $t$ is  $y(t) \in \mathbb{R}$. The transfer function is
\[
H(s) = C(sE - A)^{-1}B\,,\qquad s \in \mathbb{C}\,.
\]
In the following, we only consider systems with full-rank matrix $E$.

\subsection{Loewner rational interpolation}
\label{sec:Prelim:Loewner}
To derive a reduced model of dimension $\nr \in \mathbb{N}$ of system \eqref{eq:Prelim:FOM}, consider $2\nr$ interpolation points $s_1, \dots, s_{2\nr} \in \mathbb{C}$. The interpolation points are partitioned into two sets $\{\mu_1, \dots, \mu_{\nr}\}$ and $\{\gamma_1, \dots, \gamma_{\nr}\}$ of equal size. Define the $\nr \times \nr$ Loewner matrix $L$ and the $\nr \times \nr$ shifted Loewner matrix $L^{(s)}$ as
\[
L_{ij} = \frac{H(\mu_i) - H(\gamma_j)}{\mu_i - \gamma_j}\,,\qquad \sL_{ij} = \frac{\mu_i H(\mu_i) - \gamma_jH(\gamma_j)}{\mu_i - \gamma_j}\,,\qquad i,j = 1, \dots, \nr\,,
\]
together with the $\nr \times 1$ input matrix $\hatB$ and the $1 \times \nr$ output matrix $\hatC$ with components
\begin{equation}
\hatB_i = H(\mu_i)\,,\qquad \hatC_i = H(\gamma_i)\,,\qquad i = 1, \dots, \nr\,.
\label{eq:DefOfBHatAndCHat}
\end{equation}
The Loewner model is
\[
\hatE\dot{\hatx}(t) = \hatA \hatx(t) + \hatB u(t)\,,\qquad \haty(t) = \hatC \hatx(t)
\]
with $\hatE = -L$ and $\hatA = -\sL$. The $\nr$-dimensional state at time $t$ is $\hat{x}(t)$ and the output at time $t$ is $\hat{y}(t)$. The transfer function of the Loewner model is
\[
\hatH(s) = \hatC (s \hatE - \hatA)^{-1}\hatB\,,\qquad s \in \mathbb{C}\,.
\]
The Loewner approach guarantees that the transfer function $\hatH$ of the Loewner model interpolates the transfer function $H$ of the system \eqref{eq:Prelim:FOM} at the interpolation points $s_1, \dots, s_{2\nr}$, which means
\[
H(s_i) = \hatH(s_i)\,,\qquad i = 1, \dots, 2\nr\,.
\]

\section{Learning Loewner models from noisy frequency-response measurements}
We now study the robustness of the Loewner approach to noise in the transfer-function values of the system \eqref{eq:Prelim:FOM}. The key contribution is Theorem~\ref{thm:Noise} that bounds the error that is introduced by noise under certain conditions.

\subsection{Noisy transfer-function values}
\label{sec:Noise:Random}
Let $\mu \in \mathbb{C}$ with the real part $\Real(\mu)$ and the imaginary part $\Imag(\mu)$ and let $0 < \sigma \in \mathbb{R}$. We denote with $\epsilon \sim \Ccal\Ncal(\mu, \sigma)$ a complex random variable, where the real $\Real(\epsilon)$ and the imaginary part $\Imag(\epsilon)$ are independently normally distributed. The real part $\Real(\epsilon)$ has mean $\Real(\mu)$, the imaginary part $\Imag(\epsilon)$ has mean $\Imag(\mu)$. The real and the imaginary part of $\epsilon$ have standard deviation $\sigma$.

Let $\epsilon_1, \dots, \epsilon_{\nr} \sim \Ccal\Ncal(0, 1)$ and $\eta_1, \dots, \eta_{\nr} \sim \Ccal\Ncal(0, 1)$ be independent random variables. Define the noisy transfer-function values as
\[
H_{\sigma}(\mu_i) = H(\mu_i)(1 + \sigma\epsilon_i)\,,\qquad H_{\sigma}(\gamma_i) = H(\gamma_i)(1 + \sigma\eta_i)\,,\quad i = 1, \dots, \nr ,
\]
so that
\[
H_{\sigma}(\mu_i) \sim \Ccal\Ncal(H(\mu_i), H(\mu_i)\sigma)\,,\qquad H_{\sigma}(\gamma_i) \sim \Ccal\Ncal(H(\gamma_i), H(\gamma_i)\sigma)\,,\quad i = 1, \dots, \nr .
\]
The noise pollutes the transfer-function values in a relative sense, i.e., the standard deviation of $H_{\sigma}(\mu_i)$ is scaled by $H(\mu_i)$. We consider such a relative noise model to be realistic in our situation because measurement errors typically are relative to the value of the quantity that is measured. Note, however, that the technics used in the following analysis can be extended to an absolute noise model, where the standard deviation of the noise is independent of the transfer-function value; see Section~\ref{sec:Remarks}.

\subsection{Loewner and noisy transfer-function values}
From the noisy transfer-function values, we derive the noisy Loewner matrices
\[
\tL_{ij} = \frac{H_{\sigma}(\mu_i) - H_{\sigma}(\gamma_j)}{\mu_i - \gamma_j}\,,\qquad \tsL_{ij} = \frac{\mu_i H_{\sigma}(\mu_i) - \gamma_j H_{\sigma}(\gamma_j)}{\mu_i - \gamma_j}\,,\qquad i,j = 1, \dots, \nr\,,
\]
which have the same structure as in Section~\ref{sec:Prelim:Loewner}, except that the noisy transfer-function values $H_{\sigma}(\mu_1), \dots, H_{\sigma}(\mu_{\nr})$, $ H_{\sigma}(\gamma_1), \dots, H_{\sigma}(\gamma_{\nr})$ are used rather than the noiseless values $H(\mu_1), \dots, H(\mu_{\nr}), H(\gamma_1), \dots, H(\gamma_{\nr})$. We decompose the noisy Loewner and the noisy shifted Loewner matrix into deterministic and random parts as
\[
\tL = L + \sigma \delta L\,,\qquad \tsL = \sL + \sigma \delta \sL\,,
\]
with
\[
\delta L_{ij} = \frac{H(\mu_i)\epsilon_i - H(\gamma_j)\eta_j}{\mu_i - \gamma_j}\,,\qquad \delta \sL_{ij} = \frac{\mu_i H(\mu_i)\epsilon_i - \gamma_j H(\gamma_j)\eta_j}{\mu_i - \gamma_j}\,,\qquad i, j = 1, \dots, \nr\,.
\]
We obtain the system matrices
\[
\tE = \hatE + \sigma \delta E\,,\qquad \tA = \hatA + \sigma \delta A\,,
\]
where $\delta E = -\delta L_{ij}$ and $\delta A = -\delta \sL_{ij}$. Similarly, we  define $\tB = \hatB + \sigma\delta B$ and $\tC = \hatC + \sigma\delta C$ with
\begin{equation}\label{eq:dB-dC}
\delta B_{i} = H(\mu_i)\epsilon_i\,,\qquad \delta C_i = H(\gamma_i)\eta_i\,,\qquad i = 1, \dots, \nr.
\end{equation}
The Loewner model learned from the noisy transfer-function values is then given by
\begin{equation}
\tE\dot{\tx}(t) = \tA\tx(t) + \tB u(t)\,,\qquad \ty(t) = \tC \tx(t) ,
\label{eq:Noise:NoisyLoewnerModel}
\end{equation}
with the $\nr$-dimensional state $\tx(t)$ at time $t$ and the output $\ty(t)$ at time $t$. The transfer function of the model \eqref{eq:Noise:NoisyLoewnerModel} is
\[
\tH(s) = \tC(s\tE - \tA)^{-1}\tB\,,\qquad s \in \mathbb{C}\,.
\]
We call \eqref{eq:Noise:NoisyLoewnerModel} the noisy Loewner model.

\subsection{Noise structure}
The Loewner and the shifted Loewner matrices introduce a structure in the noise that is added to the system matrices of the Loewner model learned from the noisy transfer-function values. Consider the matrix $s \delta E - \delta A$ and obtain
\begin{align*}
s \delta E_{ij} - \delta A_{ij} = & -\frac{1}{\mu_i - \gamma_j} \left(s H(\mu_i)\epsilon_i - s H(\gamma_j)\eta_j - \mu_i H(\mu_i)\epsilon_i + \gamma_j H(\gamma_j)\eta_j\right)\\
= & \frac{1}{\gamma_j - \mu_i}\left((s - \mu_i) H(\mu_i)\epsilon_i + (-s + \gamma_j) H(\gamma_j)\eta_j\right)
\end{align*}
to write in matrix form as
\begin{multline}
s\delta E - \delta A = \underbrace{\begin{bmatrix}
\epsilon_1 & & & \\
& \epsilon_2 & & \\
& & \ddots & \\
& & & \epsilon_{\nr}\end{bmatrix}}_{\boldsymbol{\epsilon}}\underbrace{\begin{bmatrix} H(\mu_1)\frac{s - \mu_1}{\gamma_1 - \mu_1} & \dots &  H(\mu_1)\frac{s - \mu_{1}}{\gamma_1 - \mu_{\nr}}\\
\vdots & \ddots & \vdots\\
 H(\mu_{\nr})\frac{s - \mu_{\nr}}{\gamma_1 - \mu_1} & \dots &  H(\mu_{\nr})\frac{s - \mu_{\nr}}{\gamma_{\nr} - \mu_{\nr}}
\end{bmatrix}}_{F_E}  + \\
\underbrace{\begin{bmatrix} H(\gamma_1)\frac{\gamma_1 - s}{\gamma_1 - \mu_1} & \dots &  H(\gamma_{\nr})\frac{\gamma_{\nr} - s}{\gamma_1 - \mu_{\nr}}\\
\vdots & \ddots & \vdots\\
 H(\gamma_1)\frac{\gamma_1 - s}{\gamma_1 - \mu_1} & \dots &  H(\gamma_{\nr})\frac{\gamma_{\nr} - s}{\gamma_{\nr} - \mu_{\nr}}\end{bmatrix}}_{F_A}\underbrace{\begin{bmatrix}\eta_1 & & & \\
& \eta_2 & & \\
& & \ddots & \\
& & & \eta_{\nr}
\end{bmatrix}}_{\boldsymbol \eta}\,.
\label{eq:sEADeltaMat}
\end{multline}
Equation \eqref{eq:sEADeltaMat} reveals that the random parts of $s\delta E - \delta A$ can be singled out into the two diagonal random matrices $\bfepsilon$ and $\bfeta$ of dimension $\nr \times \nr$. The diagonal entries of $\bfepsilon$ and $\bfeta$ are independent and have distribution $\Ccal\Ncal(0, 1)$.

\subsection{Bounding the error due to noisy transfer-function values}
	The task is to bound
	\begin{equation}\label{H-H}
	\hatH(s) - \tH(s) = \hatC (s\hatE - \hatA)^{-1}\hatB - \tC (s\tE - \tA)^{-1}\tB,
	\end{equation}
	where $s$ typically takes the values in some specified $\Omega\subset\mathbb{C}$; for instance on the imaginary axis, possibly within a specified frequency range. In any case, we assume in the following that $\Omega$ is free of the poles of $\hatH$. The following main results hold pointwise and therefore are agnostic to additional structure imposed by $\Omega$ besides that it excludes poles of $\hatH$.

	If $\widehat{\lambda}_i$, $\widehat{\phi}_i$ ($\widetilde{\lambda}_i$, $\widetilde{\phi}_i$) are the poles with the corresponding residues of $\hatH$ ($\tH$), then \cite[Proposition 3.3]{gugercin_2008} gives
	\begin{equation}
	\| \hatH - \tH\|_{\mathcal{H}_2}^2 = \sum_{i=1}^{n} \widehat{\phi}_i (\hatH(-\widehat{\lambda}_i) - \tH(-\widehat{\lambda}_i)) +
	\sum_{j=1}^{n} \widetilde{\phi}_j (\tH(-\widetilde{\lambda}_j) - \hatH(-\widetilde{\lambda}_j))\,,
	\label{eq:ErrorH2}
	\end{equation}
	which shows that the error (\ref{H-H}) is of particular interest at the reflected poles of both systems (for \eqref{eq:ErrorH2} to hold, both $\hatH$ and $\tH$ are assumed stable, of the same order $n$, and all poles of both systems are assumed simple). We do not tackle the issue of a probabilistic error bound for the $\mathcal{H}_2$ norm; this topic is left for our future work.

	For an estimate of the error (\ref{H-H}), we need to understand the effect of the random noise in the matrices $\tE = \hatE + \sigma \delta E$, $\tA = \hatA + \sigma \delta A$,  $\tB = \hatB + \sigma \delta B$ to the solution of the linear system $(s\tE - \tA)^{-1}\tB$. To that end, we first briefly review the (deterministic) error bound for a solution of the perturbed system. The key is the condition number $\kappa_2(s \hatE - \hatA) = \|(s\hatE - \hatA)^{-1}\|_2 \|s\hatE - \hatA\|_2$, see, e.g., \cite[Theorem~7.2]{HighamBook}.
	\begin{proposition}\label{PROP:Lin-Sys-Pert}
Let $s$ be different from the poles of $\hatH$, and consider
the perturbations (\ref{eq:sEADeltaMat}) and (\ref{eq:dB-dC}) deterministic and bounded as
$\|\sigma (s\delta E - \delta A)\|_2 \leq \zeta \|s\hatE - \hatA\|_2$, $\|\sigma\delta B\|_2 \leq\zeta \|\hatB\|_2$,  where
$\zeta > 0$ is such that $ \zeta \kappa_2(s\hatE-\hatA)<1$.
Then $s\tE - \tA$ is nonsingular and
			\begin{equation}
			\frac{\|(s\tE - \tA)^{-1}\tB - (s\hatE - \hatA)^{-1}\hatB\|_2}{\|(s\hatE - \hatA)^{-1}\hatB\|_2} \leq \frac{2\zeta}{1 - \zeta \kappa_2(s\hatE - \hatA)}\kappa_2(s\hatE - \hatA)\,.
			\label{eq:Proof:SystemOfLinearPert}
			\end{equation}
	\end{proposition}
	This is the standard perturbation bound for the linear system $\widehat{G} \widehat{x}=\widehat{B}$, $\widehat{G} = s\hatE - \hatA$, under the deterministic perturbations $\Delta \widehat{G} = \sigma (s\delta E - \delta A)$ and $\Delta\hatB=\sigma\delta B$.

	Recall that by (\ref{eq:sEADeltaMat}), $s\delta E - \delta A = \bfepsilon F_E + F_A \bfeta$, where $\bfepsilon$ and $\bfeta$ are diagonal matrices whose diagonals are random vectors. These can be bounded in a probabilistic sense, using concentration inequalities which we briefly review next.

	\begin{proposition}\label{PROP:Concetration-Ineq}
	 Let $Z = [z_1, \dots, z_{\nr}]^T$ be a random vector with independent standard normal real components $z_i \sim \Ncal(0, 1), i = 1, \dots, \nr$. Then  with probability at least $1 - \exp(-\nr/2)$
	\begin{equation}
	\|Z\|_2 \leq 2\sqrt{\nr}.
	\label{eq:Helper2Real}
	\end{equation}
	 If  $Z = [z_1, \dots, z_{\nr}]^T$ is  a vector of independent complex random variables $z_i \sim \Ccal\Ncal(0, 1)$, with $\Ccal\Ncal$ defined in Section~\ref{sec:Noise:Random}, then
	 \begin{equation}
	 \|Z\|_2 \leq 4\sqrt{\nr}
	 \label{eq:Helper3}
	 \end{equation}
	 holds with probability at least $1 - 2\exp(-\nr/2)$.
	\end{proposition}

\noindent The estimate (\ref{eq:Helper2Real}) follows from Gaussian concentration and because the $\chi^2$ distribution is sub-Gaussian; see, e.g., \cite[Example~2.28]{wainwright_2019}. The statement (\ref{eq:Helper3}) follows from \eqref{eq:Helper2Real}, because $\|Z\|_2 \leq \|\Real(Z)\|_2 + \|\Imag(Z)\|_2$, where $\Real(Z)$ is the vector that has as components the real parts of the components of $Z$ and $\Imag(Z)$ is the vector that has as components the imaginary parts of $Z$.

		\begin{lemma}\label{LEM:Lin-Sys-Pert-Prob} Let $s\in\Omega$, i.e., $s$ is different from the poles of $\hatH$,  be such that
			\begin{equation}
			0 < \sigma < \frac{1}{\kappa_2(s\hatE - \hatA)} \min\left\{\frac{\|s\hatE - \hatA\|_2}{4\sqrt{\nr}\left(\|F_E\|_2 + \|F_A\|_2\right)}, \frac{\|\hatB\|_2}{4\sqrt{\nr} \|\hatB\|_{\infty}}\right\}\,.
			\label{eq:Noise:Cond}
			\end{equation}
			Then, with probability at least $1 - 4\exp(-n/2)$,  $s$ is not a pole of $\tH$  and the error bound (\ref{eq:Proof:SystemOfLinearPert}) in Proposition \ref{PROP:Lin-Sys-Pert} holds .
		\end{lemma}
	\begin{proof}	Since, by (\ref{eq:sEADeltaMat}), $s\delta E - \delta A = \bfepsilon F_E + F_A \bfeta$,
we have
		\begin{equation}
		\|s \delta E - \delta A\|_2 \leq \|\bfepsilon\|_2\|F_E\|_2 + \|F_A\|_2\|\bfeta\|_2
		\label{eq:Proof:sEABoundEpsilon}.
		\end{equation}
Because $\bfepsilon$ is diagonal\footnote{Note that $\bfepsilon$ is a diagonal \emph{matrix} defined in \eqref{eq:sEADeltaMat}.} with the elements $\epsilon_1, \dots, \epsilon_{\nr}$ on the diagonal, we obtain, using Proposition \ref{PROP:Concetration-Ineq}, that
		\begin{equation}
		\|\bfepsilon\|_2 = \max_{i = 1, \dots, \nr} |\epsilon_i| = \|[\epsilon_1, \dots, \epsilon_{\nr}]^T\|_{\infty} \leq \|[\epsilon_1, \dots, \epsilon_{\nr}]^T\|_2 \leq 4\sqrt{\nr}\,,
		\label{eq:Proof:MyBoundEpsilon}
		\end{equation}
		with probability at least $1 - 2\exp(-n/2)$. Let $\delta B_i$ denote the $i$-th component of $\delta B$, then $\delta B_i = H(\mu_i)\epsilon_i = \widehat{B}_i\epsilon_i$ holds for $i = 1, \dots, \nr$ and thus equation \eqref{eq:Proof:MyBoundEpsilon} means that
		\begin{eqnarray}
		\|\delta B\|_2 & \leq& \|\hatB\|_{\infty}\|[\epsilon_1, \dots, \epsilon_n]^T\|_2 \leq 4\sqrt{\nr}\|\hatB\|_{\infty}
		\label{eq:Proof:MyBoundB} \\
		\|\tB \|_2 &\leq & (1+\sigma 4 \sqrt{n}) \| \hatB\|_2 \label{eq:Proof:MyBoundB-1}
		\end{eqnarray}
		holds with probability at least $1 - 2\exp(-n/2)$. Similar arguments show that
		\begin{equation}
		\|\bfeta\|_2 \leq 4\sqrt{\nr}
		\label{eq:Proof:MyBoundEta}
		\end{equation}
		and
		\begin{equation}
		\|\delta C\|_2 \leq 4 \sqrt{\nr} \|\hatC\|_{\infty}
		\label{eq:Proof:MyBoundC}
		\end{equation}
		hold with probability at least $1 - 2\exp(-n/2)$, where we used $\delta C_i = H(\gamma_i)\eta_i = \widehat{C}_i\eta_i$ for $i = 1, \dots, \nr$.
		Set now
		\begin{equation}
		\zeta = \sigma \widehat{\zeta}, \;\; \widehat{\zeta} = \max\left\{\frac{4\sqrt{\nr}\left(\|F_E\|_2 + \|F_A\|_2\right)}{\|s\hatE - \hatA\|_2}, \frac{4\sqrt{\nr} \|\hatB\|_{\infty}}{\|\hatB\|_2}\right\}
		\label{eq:Proof:MyZeta}
		\end{equation}
		and observe that \eqref{eq:Noise:Cond} guarantees $\zeta \kappa_2(s\hatE - \hatA) < 1$.
		Thus, together with \eqref{eq:Proof:sEABoundEpsilon}, it follows that
		\begin{multline}
		\|\sigma(s \delta E - \delta A)\|_2 \leq \sigma 4\sqrt{\nr}(\|F_E\|_2 + \|F_A\|_2) \\ = \sigma \frac{4\sqrt{\nr}(\|F_E\|_2 + \|F_A\|_2)}{\|s \hatE - \hatA\|_2} \|s \hatE - \hatA\|_2
		\leq \zeta \|s \hatE - \hatA\|_2
		\label{eq:Proof:CondDetBoundA}
		\end{multline}
		with probability at least $(1 - 2\exp(-n/2))^2 \geq 1 - 4\exp(-n/2)$, where we used \eqref{eq:Proof:MyZeta}.

		With \eqref{eq:Proof:MyBoundB} and the definition of $\zeta$ in \eqref{eq:Proof:MyZeta}, we also obtain that
		\begin{equation}
		\|\sigma \delta B\|_2 \leq \sigma 4\sqrt{\nr} \|\hatB\|_{\infty} = \sigma \frac{4\sqrt{\nr}\|\hatB\|_{\infty}}{\|\hatB\|_2}\|\hatB\|_2 \leq \zeta \|\hatB\|_2
		\label{eq:Proof:CondDetBoundB}
		\end{equation}
		holds with probability at least $1 - 4\exp(-n/2)$.
		Thus, with \eqref{eq:Proof:CondDetBoundA}, \eqref{eq:Proof:CondDetBoundB}, and because \eqref{eq:Noise:Cond} implies $\zeta \kappa_2(s\hatE - \hatA) < 1$, the error bound (\ref{eq:Proof:SystemOfLinearPert})
			%\eqref{eq:Proof:DetBound}
		is applicable with probability at least $1 - 4\exp(-n/2)$, which also means that $s\tE - \tA$ is nonsingular.
		\end{proof}

	The following theorem bounds the error due to noise in the transfer-function values.
	\begin{theorem} Under the same assumptions as Lemma \ref{LEM:Lin-Sys-Pert-Prob}, for each $s\in\Omega$, there exists a constant $C_s > 0$ that may depend on $s$ such that
		\begin{equation}
		|\tH(s) - \hatH(s)| \leq C_s \sigma
		\label{eq:NoiseBound}
		\end{equation}
		holds with probability at least $1 - 4\exp(-n/2)$.
		\label{thm:Noise}
	\end{theorem}
\begin{proof}
Consider now
\begin{align*}
	& \hatH(s) - \tH(s) =  \hatC (s\hatE - \hatA)^{-1}\hatB - \tC (s\tE - \tA)^{-1}\tB\\
	& =  \hatC  \left((s\hatE - \hatA)^{-1}\hatB - (s\tE- \tA)^{-1}\tB\right) - \sigma\delta C (s\tE - \tA)^{-1}\tB
\end{align*}
		and take the absolute value to obtain
		\begin{equation}
	|\hatH(s) - \tH(s)| \leq
	c_1 \frac{\zeta}{1 - \zeta \kappa_2(s \hatE - \hatA)} + |\sigma\delta C (s\tE - \tA)^{-1}\tB|  ,
	\label{eq:Proof:FirstBoundOnH}
	\end{equation}
		where we invoked \eqref{eq:Proof:SystemOfLinearPert} with
		\begin{equation}
		c_1 = 2\|\hatC\|_2 \|(s\hatE - \hatA)^{-1}\hatB\|_2 \kappa_2(s \hatE - \hatA)
		\label{eq:C1}
		\end{equation}
		 and $\zeta$ set as in \eqref{eq:Proof:MyZeta}. Note that \eqref{eq:Proof:SystemOfLinearPert} holds with probability at least $1 - 4\exp(-\nr/2)$, and thus \eqref{eq:Proof:FirstBoundOnH} holds with the same probability.

		We now bound $\|(s\tE - \tA)^{-1}\|_2$ in probability. Consider the Neumann expansion
		\begin{multline}
		(s\tE - \tA)^{-1} = \left(s\hatE - \hatA + \sigma\left(s\delta E - \delta A\right)\right)^{-1} \\
		= (s \hatE - \hatA)^{-1}\sum_{i = 0}^{\infty} (-1)^i \sigma^i \left((s \delta E - \delta A)(s\hatE - \hatA)^{-1}\right)^i\,,\label{eq:Neumann-series}
		\end{multline}
	where the series  converges to the inverse of $I +  \sigma (s \delta E - \delta A)(s\hatE - \hatA)^{-1}$ provided that\footnote{The necessary and sufficient condition for the convergence is that the spectral radius of $\sigma (s \delta E - \delta A)(s\hatE - \hatA)^{-1}$  is strictly less than one.} $\|\sigma (s \delta E - \delta A)(s\hatE - \hatA)^{-1}\|_2 < 1$. Because \eqref{eq:Proof:CondDetBoundA} holds with probability at least $1 - 4\exp(-\nr/2)$, we obtain that with the same probability of at least $1 - 4\exp(-\nr/2)$
		\begin{displaymath}
		\|\sigma (s \delta E - \delta A)(s\hatE - \hatA)^{-1}\|_2
		\leq \zeta {\| s\hatE - \hatA \|_2}{\|(s\hatE - \hatA)^{-1}\|_2} < 1,
		\end{displaymath}
		 holds, where we used assumption \eqref{eq:Noise:Cond} in the second inequality. Note that the second inequality is strict.
		 Set
		\[
		\nu = 4\sqrt{\nr}(\|F_E\|_2 + \|F_A\|_2)\|(s\hatE - \hatA)^{-1}\|_2\,,
		\]
		and obtain with (\ref{eq:Proof:CondDetBoundA}), (\ref{eq:Neumann-series}), and $\sigma\nu < 1$ because of (\ref{eq:Noise:Cond}) that
		\begin{equation}
		\|(s\tE - \tA)^{-1}\|_2 \leq \|(s\hatE - \hatA)^{-1}\|_2 \sum_{i = 0}^{\infty}  (\nu\sigma)^i = \|(s\hatE - \hatA)^{-1}\|_2 \frac{1}{1 - \nu\sigma}
		\label{eq:Proof:sEABound}
		\end{equation}
		holds with probability at least $1 - 4\exp(-\nr/2)$.

		Then, we obtain the bound
		\begin{align}
		|\sigma \delta C (s\tE - \tA)^{-1} \tB| \leq & \sigma \|\delta C\|_2 \|(s \tE - \tA)^{-1} \|_2 \|\tB\|_2\notag\\
		\leq & 4 \sqrt{\nr} \|\hatC\|_{\infty}\|\hatB\|_2 \|(s\hatE - \hatA)^{-1}\|_2 \frac{\sigma (1+\sigma 4 \sqrt{n})}{1 - \nu\sigma}\notag\\
		\leq & c_2 \frac{\sigma+\sigma^2 4 \sqrt{n}}{1 - \nu \sigma} , \label{eq:Proof:BoundOnHTerm2-1}
		\end{align}
		where we used \eqref{eq:Proof:MyBoundB-1}, \eqref{eq:Proof:MyBoundC} and \eqref{eq:Proof:sEABound}, which together hold with probability at least $1 - 4\exp(-\nr/2)$, and we set $c_2 = 4 \sqrt{\nr} \|\hatC\|_{\infty}\|\hatB\|_2 \|(s\hatE - \hatA)^{-1}\|_2$.

		We obtain with \eqref{eq:Proof:FirstBoundOnH}, \eqref{eq:Proof:BoundOnHTerm2-1} the following bound
		\begin{equation}
		|\hatH(s) - \tH(s)| \leq  \sigma \left[ \frac{c_1 \widehat{\zeta}}{1 -\sigma\widehat{ \zeta }\kappa_2(s \hatE - \hatA)} +  \frac{c_2(1+\sigma 4\sqrt{n})}{1 - \nu \sigma}\right]\,,
		\label{eq:Proof:BoundOnHSecond}
		\end{equation}
		which grows at most linearly in $\sigma$ as long as condition \eqref{eq:Noise:Cond} is satisfied, which shows \eqref{eq:NoiseBound}.
		\end{proof}

\begin{corollary}
Under the same conditions as Theorem~\ref{thm:Noise} and for $s\in\Omega$ that are not zeros of $\hatH$, 		\begin{equation}
		\frac{|\tH(s) - \hatH(s)|}{|\hatH(s)|} \leq C_s \sigma
		\label{eq:NoiseBoundRelative}
		\end{equation}
		holds with probability at least $1 - 4\exp(-n/2)$ and constant $C_s > 0$ that may depend on $s$.
\end{corollary}
\begin{proof}
Note that $\hatH(s)$ is independent of $\sigma$ and thus dividing \eqref{eq:NoiseBound} by $|\hatH(s)|$ is sufficient to show \eqref{eq:NoiseBoundRelative} if $s$ is not a zero of $\hatH$. To highlight a geometric interpretation of \eqref{eq:NoiseBoundRelative}, we show \eqref{eq:NoiseBoundRelative} via a different approach.
	Since $|\hatH(s)| = \|\hatC\|_2 \|(s \hatE - \hatA)^{-1}\hatB\|_2 |\cos\angle(\hatC^*,(s \hatE - \hatA)^{-1}\hatB)|$, the factor $c_1$ defined in \eqref{eq:C1} can be interpreted  as
	$$
	c_1 = |\hatH(s)| \frac{2\kappa_2(s \hatE - \hatA)}{|\cos\angle(\hatC^*,(s \hatE - \hatA)^{-1}\hatB)|}
	$$
	so that the first term on the right-hand side in \eqref{eq:Proof:BoundOnHSecond} contains the bound on the relative error $|\hatH(s) - \tH(s)|/|\hatH(s)|$ with the two natural condition numbers $\kappa_s(s\hatE - \hatA)$ and $|\cos\angle(\hatC^*,(s \hatE - \hatA)^{-1}\hatB)|$. The interpretation is as follows: Evaluating $\tH$ essentially means solving a perturbed linear system $(s\tE - \tA)^{-1}\tB$ and then computing an inner product $(\tC^*)^*((s\tE - \tA)^{-1}\tB)$ with a perturbed vector $\tC^*$. The sensitivity of the solution of a system of equations to perturbations is quantified by the condition number $\kappa_2(s\hatE - \hatA)$ and the sensitivity of the inner product is quantified by $|\cos\angle(\hatC^*,(s \hatE - \hatA)^{-1}\hatB)|$. Similarly, the second term on the right-hand side in  \eqref{eq:Proof:BoundOnHSecond} can be modified as follows:
    Instead of \eqref{eq:Proof:BoundOnHTerm2-1}, estimate the second term in \eqref{eq:Proof:FirstBoundOnH} as
    \begin{multline}
     |\sigma \delta C (s\tE - \tA)^{-1} \tB| \leq \sigma 4\sqrt{n} \|\hatC\|_{\infty} \| (s \hatE - \hatA)^{-1}\hatB\|_2 \left(1 + \frac{2\zeta\kappa_2(s\hatE - \hatA) }{1 - \zeta \kappa_2(s\hatE - \hatA)}\right) \nonumber \\
    = \sigma |\hatH(s)| \frac{4\sqrt{n}}{|\cos\angle(\hatC^*,(s \hatE - \hatA)^{-1}\hatB)|}\frac{1+\zeta\kappa_2(s\hatE - \hatA) }{1 - \zeta \kappa_2(s\hatE - \hatA)} \frac{\|\hatC\|_{\infty}}{\|\hatC\|_{2}}\,.
    \end{multline}
    Hence, because $\hatH(s)\neq 0$ per assumption, we can write
    \begin{equation}
        \frac{|\hatH(s) - \tH(s)|}{|\hatH(s)|} \leq \sigma\left(
        \frac{1}{|\cos\vartheta_s|}\left( \frac{2\widehat{\zeta}\kappa_2^{(s)}}{1-\sigma\widehat{\zeta}\kappa_2^{(s)}} + 4\sqrt{n}\frac{1+\zeta \kappa_2^{(s)}}{1 - \zeta \kappa_2^{(s)}} \frac{\|\hatC\|_{\infty}}{\|\hatC\|_{2}} \right)\right)
    \end{equation}
    where $\vartheta_s = \angle(\hatC^*,(s \hatE - \hatA)^{-1}\hatB)$ and $\kappa_2^{(s)} = \kappa_2(s\hatE - \hatA)$.
	\end{proof}

\subsection{Remarks on Theorem~\ref{thm:Noise}}
\label{sec:Remarks}
The following remarks are in order. First, note that condition \eqref{eq:Noise:Cond} implies that $\tH(s)$ exists with probability at least $1 - 4\exp(-\nr/2)$. Second, a similar result as in Theorem~\ref{thm:Noise} can be shown for an absolute noise model, i.e., where the noisy transfer-function values $\tH(s)$  have a standard deviation that is independent of the transfer-function value $H(s)$. The major change is that matrices $F_A$ and $F_E$ defined in \eqref{eq:sEADeltaMat} are independent of the transfer-function values when an absolute noise model is used and thus condition \eqref{eq:Noise:Cond} holds for a different range of standard deviations $\sigma$ than for a relative noise model. Third, condition \eqref{eq:Noise:Cond} in Theorem~\ref{thm:Noise} depends on the scaling of the entries of the system matrices $\hatE, \hatA, \hatB$, and $\hatC$. Consider two nonsingular matrices $D_1$ and $D_2$ of size $\nr \times \nr$. Then, the Loewner model given by $\hatE, \hatA, \hatB, \hatC$, derived from noiseless transfer-function values, can be transformed as
	\[
	\breve{E} = D_1\hatE D_2,\quad \breve{A} = D_1 \hatA D_2,\quad \breve{B} = D_1 \hatB,\quad \breve{C} = \hatC D_2\,,
	\]
	with transfer function $\breve{H}(s) = \breve{C}(s\breve{E} - \breve{A})^{-1}\breve{B}$. It holds $\breve{H}(s) = \hatH(s)$ for $s \in \mathbb{C}$; however, the condition number $\kappa_2(s\hatE - \hatA)$ is \emph{not} invariant under the transformation given by $D_1$ and $D_2$, which means that condition \eqref{eq:Noise:Cond} in Theorem~\ref{thm:Noise} is not invariant if the system is transformed. A component-wise analysis \cite[Section~7.2]{HighamBook} could be an option to derive a version of Theorem~\ref{thm:Noise} that is invariant under diagonal linear transformations $D_1$ and $D_2$. Fourth, condition \eqref{eq:Noise:Cond} in Theorem~\ref{thm:Noise} depends on the interpolation points $s_1, \dots, s_{2\nr}$ and on the partition $\{\mu_1, \dots, \mu_{\nr}\}, \{\gamma_1, \dots, \gamma_{\nr}\}$, which shows that the interpolation points and their partition can influence the robustness of the Loewner approach to noise; cf., e.g., \cite[Section~2.1]{CosminsThesis}. Our numerical results demonstrate that different choices of interpolation points indeed influence condition \eqref{eq:Noise:Cond} and thus when the linear growth of the error \eqref{eq:NoiseBound} with the standard deviation $\sigma$ of the noise holds. The numerical results below suggest that taking interpolation points that keep the condition number of $s \hatE - \hatA$ low seems reasonable. However, additional analyses and numerical experiments are necessary before one can give a definitive recommendation. The issue of selection and partition of the interpolation points is an important challenging problem that will be addressed in our future work.

\section{Numerical results}
In this section, we demonstrate the error bound of  Theorem~\ref{thm:Noise} on numerical experiments with some well known benchmark examples.

\subsection{CD player}
The system of a CD player is a common benchmark problem for model reduction and can be downloaded from the SLICOT website \footnote{http://slicot.org/20-site/126-benchmark-examples-for-model-reduction}.

\subsubsection{Problem setup}
\label{sec:NumExp:CDPlayer:Setup}
We consider single-input-single-output (SISO) systems and therefore we use only the first input and the second output of the CD-player system. The frequency range is $[2\pi, 200\pi]$, which contains some of the major dynamics of the CD-player system. The order of the system is $\nh = 120$. To derive a Loewner model of order $\nr$, we select $2\nr$ interpolation points as follows. First, $\nr$ points $w_1, \dots, w_{\nr}$ are selected logarithmically equidistant in the range $[2\pi, 200\pi]$ on the imaginary axis. Then, $s_{i} = (-1)^iw_i$ for $i = 1, \dots, \nr$ and $s_{\nr + i} = (-1)^{i + 1}w_i$. The first $n$ points $s_1, \dots, s_{\nr}$ are put into the set $\{\mu_1, \dots, \mu_{\nr}\}$ and the following $n$ points $s_{\nr + 1}, \dots, s_{2\nr}$ are put into $\{\gamma_1, \dots, \gamma_{\nr}\}$. To test the Loewner models, we select 200 test points $s_1^{\text{test}}, \dots, s_{200}^{\text{test}}$ on the imaginary axis in the range $[2\pi, 200\pi]$.

From the noiseless transfer-function values $H(\mu_1), \dots, H(\mu_{\nr})$ and $H(\gamma_1), \dots, H(\gamma_{\nr})$ a Loewner model is derived with transfer function $\hatH$. Then, the transfer-function values are polluted with noise with standard deviation $\sigma$ as described in Section~\ref{sec:Noise:Random} and a noisy Loewner model is derived with transfer function $\tH_{\sigma}$. Note that we now explicitly denote standard deviation as subscript in the transfer functions of noisy Loewner models.

\begin{figure}[t]
\centering
\begin{tabular}{cc}
\resizebox{0.48\columnwidth}{!}{\LARGE\input{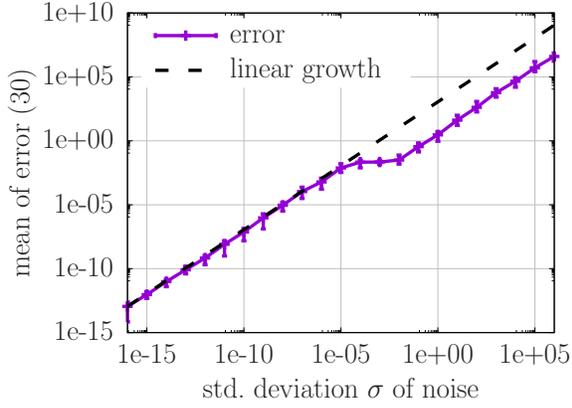}} & \resizebox{0.48\columnwidth}{!}{\LARGE\input{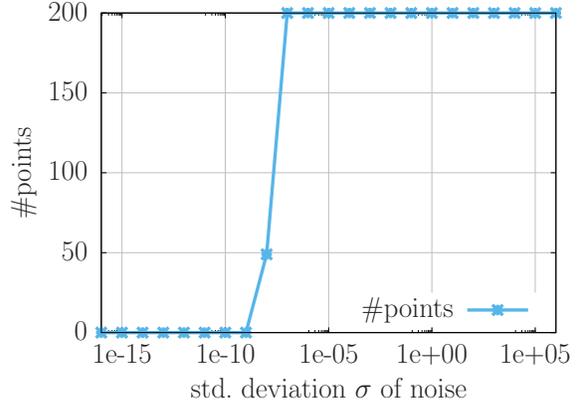}}\\
(a) error, dimension $\nr = 20$ & (b) assumption \eqref{eq:Noise:Cond} violated, dimension $\nr = 20$\\
\resizebox{0.48\columnwidth}{!}{\LARGE\input{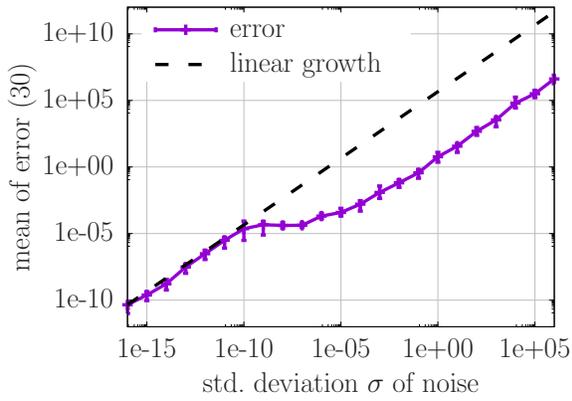}} & \resizebox{0.48\columnwidth}{!}{\LARGE\input{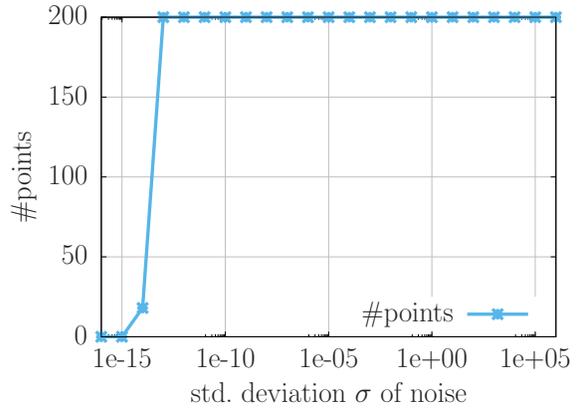}}\\
(c) error, dimension $\nr = 28$ & (d) assumption \eqref{eq:Noise:Cond} violated, dimension $\nr = 28$
\end{tabular}
\caption{CD player: Plots (a) and (c) show the growth of the mean of error \eqref{eq:NumExp:Error} over 10 replicates of independent noise samples. Plots (b) and (d) show the number of test points for which condition \eqref{eq:Noise:Cond} of Theorem~\ref{thm:Noise} is violated. The mean of the error \eqref{eq:NumExp:Error} grows linearly with $\sigma$ as long as condition \eqref{eq:Noise:Cond} is satisfied, which is in agreement with Theorem~\ref{thm:Noise}. The error bars in (a) and (c) show the minimum and maximum of the error \eqref{eq:NumExp:Error} over the 10 replicates of the noise samples.}
\label{fig:CDPlayer:Error}
\end{figure}

\subsubsection{Results}
We consider the error
\begin{equation}
e(\sigma) = \frac{1}{200}\sum_{i = 1}^{200} |\hatH(s_i^{\text{test}}) - \tH_{\sigma}(s_i^{\text{test}})|\,,
\label{eq:NumExp:Error}
\end{equation}
which is an average of the error \eqref{eq:NoiseBound} over all 200 test points. Figure~\ref{fig:CDPlayer:Error}a shows the mean of $e(\sigma)$ over 10 replicates of independent noise samples. The standard deviation $\sigma$ is in the range $[10^{-15}, 10^5]$ and the dimension is $\nr = 20$. The error bars in Figure~\ref{fig:CDPlayer:Error}a show the minimum and maximum of $e(\sigma)$ over the 10 replicates. A linear growth of the mean of error \eqref{eq:NumExp:Error} with the standard deviation $\sigma$ is observed for $\sigma < 10^{-5}$. Figure~\ref{fig:CDPlayer:Error}b shows the number of test points that violate condition \eqref{eq:Noise:Cond} of Theorem~\ref{thm:Noise}. The results indicate that for $\sigma \geq 10^{-7}$ condition \eqref{eq:Noise:Cond} is violated for all 200 test points, which seems to align with Figure~\ref{fig:CDPlayer:Error}a that shows a linear growth for $\sigma < 10^{-5}$. Thus, the results in Figure~\ref{fig:CDPlayer:Error}a are in agreement with Theorem~\ref{thm:Noise}. Similar observations can be made for $\nr = 28$ in Figure~\ref{fig:CDPlayer:Error}c and Figure~\ref{fig:CDPlayer:Error}d.

\subsection{Penzl}
The Penzl system is a benchmark problem that has been introduced in \cite[Example 3]{Penzl2006322} and is used in, e.g., \cite{CosminsThesis,PSW16TLoewner} to demonstrate the Loewner approach.

\subsubsection{Problem setup}
The Penzl system is of order $\nh = 1006$ and we consider the frequency range $[10, 1000]$. We consider two different sets of interpolation points in this example. First, we select $\nr$ logarithmically equidistant points on the imaginary axis in the range $[10, 1000]$ and include their complex conjugates as described in Section~\ref{sec:NumExp:CDPlayer:Setup}. We denote the corresponding set of $2\nr$ interpolation points as $\mathcal{E}_{\nr}$.
Second, we select points randomly in the imaginary plane in the range $[10, 1000] \times \bfiota[10, 1000]$, where $\bfiota$ is the complex unit $\bfiota = \sqrt{-1}$. To select the points randomly, we first select $10^6$ logarithmically equidistant points $w_1, \dots, w_{10^6}$ in $[10, 1000]$ and then draw uniformly $2\nr$ points from $w_1, \dots, w_{10^6}$ for the real and imaginary parts of the $\nr$ interpolation points $s_1, \dots, s_{\nr}$. Then, their complex conjugates are $s_{\nr + i} = \bar{s}_i$ for $i = 1, \dots, \nr$. The interpolation points are partitioned into the two sets $\{\mu_1, \dots, \mu_{\nr}\}$ and $\{\gamma_1, \dots, \gamma_{\nr}\}$ with $\mu_i = s_{2i - 1}$ and $\gamma_i = s_{2i}$ for $i = 1, \dots, \nr$ such that complex pairs are in the same set. The set of interpolation points is denoted as $\mathcal{R}_{\nr}$.
The test points are 200 points selected on the imaginary axis in the range $[10, 1000]$.

\subsubsection{Results}
Figure~\ref{fig:Penzl:Error}a shows the error \eqref{eq:NumExp:Error} corresponding to the Loewner model with interpolation points $\mathcal{R}_{\nr}$ (random) and of dimension $\nr = 16$. Figure~\ref{fig:Penzl:Error}c shows that condition \eqref{eq:Noise:Cond} is violated for all test points and all standard deviations $\sigma$ in the range $[10^{-15}, 10^5]$, which means that Theorem~\ref{thm:Noise} is not applicable. In contrast, Figure~\ref{fig:Penzl:Error}b shows a linear growth for the error \eqref{eq:NumExp:Error} corresponding to the models learned from logarithmically equidistant points $\mathcal{E}_{\nr}$. Figure~\ref{fig:Penzl:Error}d indicates that up to $\sigma = 10^{-10}$ the condition \eqref{eq:Noise:Cond} for Theorem~\ref{thm:Noise} is satisfied, which explains the linear growth for $\sigma \leq 10^{-10}$. For $\sigma > 10^{-10}$, Theorem~\ref{thm:Noise} is not applicable, even though a linear growth is observed, which demonstrates that Theorem~\ref{thm:Noise} is rather pessimistic in this example. Figure~\ref{fig:Penzl:TF} shows the magnitude of the transfer functions for $\sigma = 10^{-6}$ and demonstrates that the logarithmically equidistant points $\mathcal{E}_{\nr}$ seem to provide more robustness against noise than the random points $\mathcal{R}_{\nr}$ in this example. The numerical observations seem to be in alignment with Theorem~\ref{thm:Noise} because assumption \eqref{eq:Noise:Cond} is violated for interpolation points $\mathcal{R}_{\nr}$ for all $\sigma$ considered in Figure~\ref{fig:Penzl:Error}c, whereas assumption \eqref{eq:Noise:Cond} is satisfied for small $\sigma$ for points $\mathcal{E}_{\nr}$. Note that the condition number of $s \hatE - \hatA$ plays a critical role in the assumption \eqref{eq:Noise:Cond}. Thus, it seems that taking interpolation points that keep the condition number of $s \hatE - \hatA$ low is reasonable. However, additional analyses and numerical experiments are necessary before one can give a definitive recommendation.

\begin{figure}
\centering
\begin{tabular}{cc}
\resizebox{0.48\columnwidth}{!}{\LARGE\input{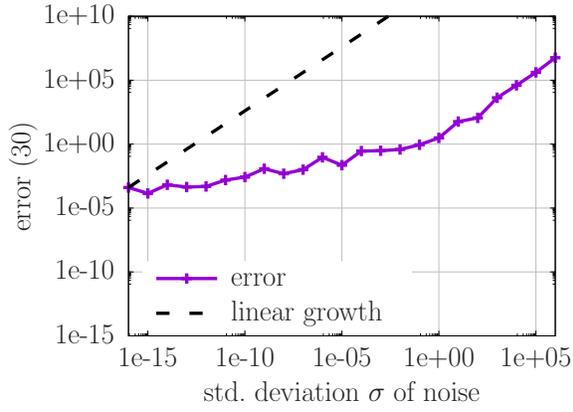}} & \resizebox{0.48\columnwidth}{!}{\LARGE\input{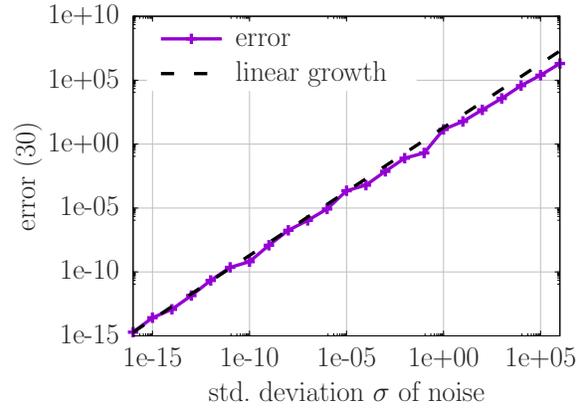}}\\
(a) error, interpolation points $\mathcal{R}_{\nr}$ & (b) error, interpolation points $\mathcal{E}_{\nr}$\\
\resizebox{0.48\columnwidth}{!}{\LARGE\input{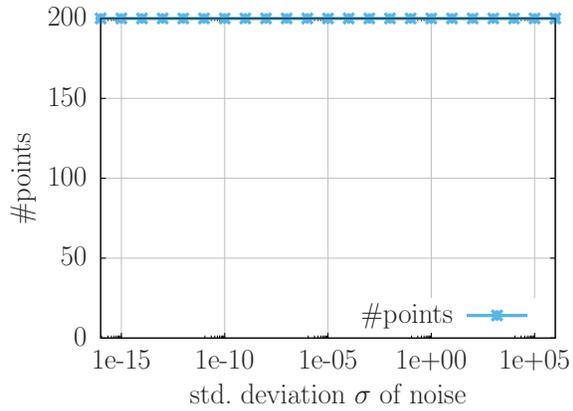}} & \resizebox{0.48\columnwidth}{!}{\LARGE\input{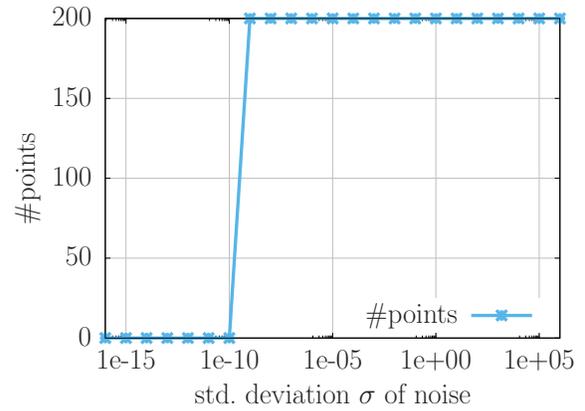}}\\
(c) assumption \eqref{eq:Noise:Cond}, interpolation points $\mathcal{R}_{\nr}$ & (d) assumption \eqref{eq:Noise:Cond}, interpolation points $\mathcal{E}_{\nr}$
\end{tabular}
\caption{Penzl: The choice of the interpolation points can have a significant effect on the robustness of the Loewner approach to noise. Plots (a) and (c) show results for points selected randomly and plots (b) and (d) for points selected logarithmically equidistantly. Learning Loewner models from the logarithmically equidistant points seems to be more robust in this example than learning from the randomly selected points.}
\label{fig:Penzl:Error}
\end{figure}

\begin{figure}
\centering
\begin{tabular}{l}
\resizebox{1\columnwidth}{!}{\LARGE\input{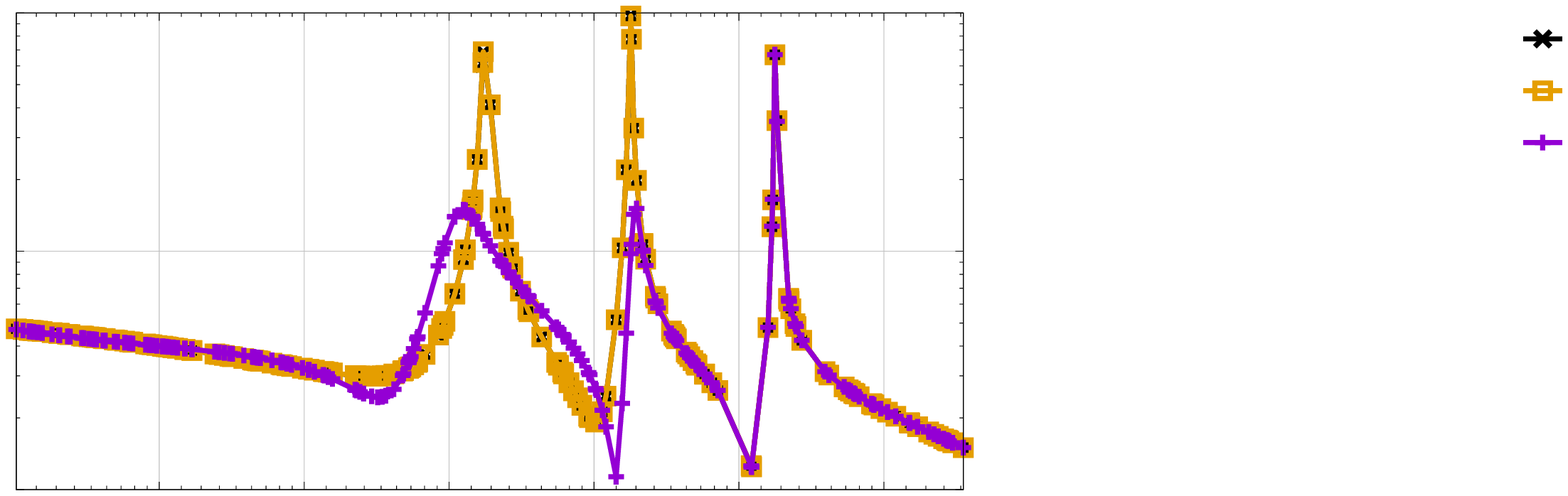}}\\
(a) Models learned from randomly selected points\\
\resizebox{1\columnwidth}{!}{\LARGE\input{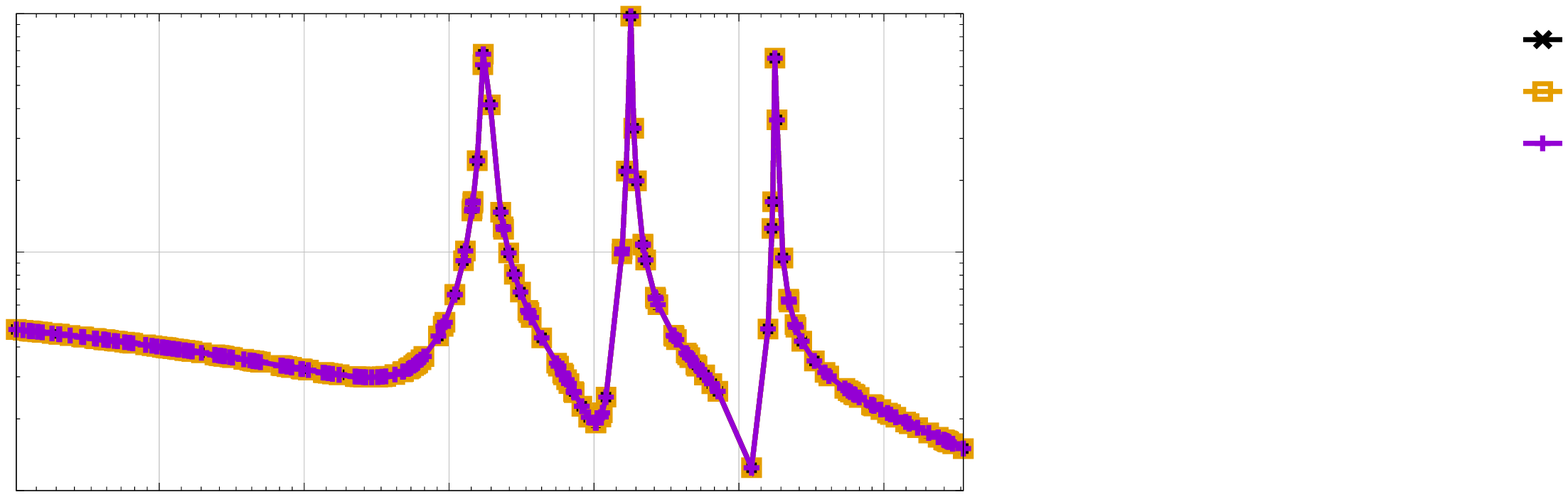}}\\
(b) Models learned from logarithmically equidistant points
\end{tabular}
\caption{Penzl: Plots (a) and (b) show the magnitude of the transfer function of Loewner models learned from noiseless transfer-function values and from transfer-function values polluted with noise with standard deviation $\sigma = 10^{-6}$. The interpolation points are randomly selected in (a) and logarithmically equidistant in (b);  compare with Figure \ref{fig:Penzl:Error}.}
\label{fig:Penzl:TF}
\end{figure}

\section*{Acknowledgements}
Drma\v{c} was supported in parts by the DARPA  Contract HR0011-18-9-0033, the ONR Contract N00014-19-C-1053,  and the Croatian Science Foundation through Grant IP-2019-04-6268 \emph{``Randomized low rank algorithms and applications to parameter dependent problems."} Peherstorfer was partially supported by US Department of Energy, Office of Advanced Scientific Computing Research, Applied Mathematics Program (Program Manager Dr. Steven Lee), DOE Award DESC0019334 and by the National Science Foundation under Grant No.~1901091. The numerical experiments were computed with support through the NYU IT High Performance Computing resources, services, and staff expertise.

\bibliography{noiseloew}

\begin{thebibliography}{10}

\bibitem{NewFIR}
K.~Abderrahim, H.~Mathlouthi, and F.~Msahli.
\newblock New approaches to finite impulse response systems identification
  using higher-order statistics.
\newblock {\em IET Signal Processing}, 4(5):488--501, 2010.

\bibitem{ANTOULAS01011986}
A.~Antoulas and B.~D.~O. Anderson.
\newblock On the scalar rational interpolation problem.
\newblock {\em IMA Journal of Mathematical Control \& Information},
  3(2-3):61--88, 1986.

\bibitem{AntBG10}
A.~Antoulas, C.~Beattie, and S.~Gugercin.
\newblock Interpolatory model reduction of large-scale dynamical systems.
\newblock In J.~Mohammadpour and K.~Grigoriadis, editors, {\em Efficient
  Modeling and Control of Large-Scale Systems}. Springer-Verlag, 2010.

\bibitem{doi:10.1137/15M1041432}
A.~Antoulas, I.~Gosea, and A.~Ionita.
\newblock Model reduction of bilinear systems in the {Loewner} framework.
\newblock {\em SIAM Journal on Scientific Computing}, 38(5):B889--B916, 2016.

\bibitem{BeaG12}
C.~Beattie and S.~Gugercin.
\newblock {Realization-independent $\mathcal{H}_2$-approximation}.
\newblock In {\em Proc. IEEE Conf. Decis. Control}, pages 4953--4958, Maui, HI,
  USA, 2012.

\bibitem{BEATTIE20122916}
C.~Beattie, S.~Gugercin, and S.~Wyatt.
\newblock Inexact solves in interpolatory model reduction.
\newblock {\em Linear Algebra and its Applications}, 436(8):2916 -- 2943, 2012.
\newblock Special Issue dedicated to Danny Sorensen's 65th birthday.

\bibitem{MORSurveySIREV}
P.~Benner, S.~Gugercin, and K.~Willcox.
\newblock A survey of projection-based model reduction methods for parametric
  dynamical systems.
\newblock {\em SIAM Review}, 57(4):483--531, 2015.

\bibitem{Brunton12042016}
S.~L. Brunton, J.~L. Proctor, and J.~N. Kutz.
\newblock Discovering governing equations from data by sparse identification of
  nonlinear dynamical systems.
\newblock {\em Proceedings of the National Academy of Sciences},
  113(15):3932--3937, 2016.

\bibitem{DDMD-SISC-2018}
Z.~{Drma{\v c}}, I.~{Mezi{\'c}}, and R.~{Mohr}.
\newblock {Data driven modal decompositions: analysis and enhancements}.
\newblock {\em SIAM Journal on Scientific Computing}, 40(4):A2253--A2285, 2018.

\bibitem{VectorFitting}
Z.~Drma\v{c}, S.~Gugercin, and C.~Beattie.
\newblock Quadrature-based vector fitting for discretized $\mathcal{H}_2$
  approximation.
\newblock {\em SIAM Journal on Scientific Computing}, 37(2):A625--A652, 2015.

\bibitem{doi:10.1137/15M1010774}
Z.~Drma\v{c}, S.~Gugercin, and C.~Beattie.
\newblock Vector fitting for matrix-valued rational approximation.
\newblock {\em SIAM Journal on Scientific Computing}, 37(5):A2346--A2379, 2015.

\bibitem{embree2019pseudospectra}
M.~Embree and A.~C. Ionita.
\newblock Pseudospectra of {Loewner} matrix pencils.
\newblock {\em arXiv:1910.12153}, pages 1--18, 2019.

\bibitem{doi:10.1002/nla.2200}
I.~V. Gosea and A.~C. Antoulas.
\newblock Data-driven model order reduction of quadratic-bilinear systems.
\newblock {\em Numerical Linear Algebra with Applications}, 25(6):e2200, 2018.

\bibitem{gugercin_2008}
S.~Gugercin, A.~Antoulas, and C.~Beattie.
\newblock $\mathcal{H}_2$ model reduction for large-scale linear dynamical
  systems.
\newblock {\em SIAM Journal on Matrix Analysis and Applications},
  30(2):609--638, Jan. 2008.

\bibitem{772353}
B.~Gustavsen and A.~Semlyen.
\newblock Rational approximation of frequency domain responses by vector
  fitting.
\newblock {\em Power Delivery, IEEE Transactions on}, 14(3):1052--1061, Jul
  1999.

\bibitem{HighamBook}
N.~J. Higham.
\newblock {\em Accuracy and Stability of Numerical Algorithms}.
\newblock SIAM, 2002.

\bibitem{2018arXiv180300043H}
J.~M. Hokanson.
\newblock A data-driven {McMillan} degree lower bound.
\newblock {\em SIAM Journal on Scientific Computing}, 42(5):A3447--A3461, 2020.

\bibitem{doi:10.1137/130914619}
A.~Ionita and A.~Antoulas.
\newblock Data-driven parametrized model reduction in the {Loewner} framework.
\newblock {\em SIAM Journal on Scientific Computing}, 36(3):A984--A1007, 2014.

\bibitem{CosminsThesis}
A.~C. Ionita.
\newblock {\em Lagrange rational interpolation and its applications to
  approximation of large-scale dynamical systems}.
\newblock PhD thesis, Rice University, 2013.

\bibitem{LoewnerNonInt}
A.~C. Ionita and A.~C. Antoulas.
\newblock {\em Matrix pencils in time and frequency domain system
  identification}, pages 79--88.
\newblock Control, Robotics and Sensors. Institution of Engineering and
  Technology, 2012.

\bibitem{8550216}
P.~{Kergus}, S.~{Formentin}, C.~{Poussot-Vassal}, and F.~{Demourant}.
\newblock Data-driven control design in the {Loewner} framework: Dealing with
  stability and noise.
\newblock In {\em 2018 European Control Conference (ECC)}, pages 1704--1709,
  June 2018.

\bibitem{KramerERZ}
B.~Kramer and S.~Gugercin.
\newblock Tangential interpolation-based eigensystem realization algorithm for
  {MIMO} systems.
\newblock {\em Mathematical and Computer Modelling of Dynamical Systems},
  22(4):282--306, 2016.

\bibitem{KPW16ControlAdaptROM}
B.~Kramer, B.~Peherstorfer, and K.~Willcox.
\newblock Feedback control for systems with uncertain parameters using
  online-adaptive reduced models.
\newblock {\em SIAM Journal on Applied Dynamical Systems}, 16(3):1563--1586,
  2017.

\bibitem{NathanBook}
J.~N. Kutz, S.~L. Brunton, B.~W. Brunton, and J.~L. Proctor.
\newblock {\em Dynamic mode decomposition: Data-driven modeling of complex
  systems}.
\newblock SIAM, 2016.

\bibitem{5356286}
S.~Lefteriu and A.~Antoulas.
\newblock A new approach to modeling multiport systems from frequency-domain
  data.
\newblock {\em Computer-Aided Design of Integrated Circuits and Systems, IEEE
  Transactions on}, 29(1):14--27, Jan 2010.

\bibitem{Lefteriu2010}
S.~Lefteriu, A.~C. Ionita, and A.~C. Antoulas.
\newblock Modeling systems based on noisy frequency and time domain
  measurements.
\newblock In J.~C. Willems, S.~Hara, Y.~Ohta, and H.~Fujioka, editors, {\em
  Perspectives in Mathematical System Theory, Control, and Signal Processing: A
  Festschrift in Honor of Yutaka Yamamoto on the Occasion of his 60th
  Birthday}, pages 365--378, Berlin, Heidelberg, 2010. Springer Berlin
  Heidelberg.

\bibitem{LjungBook}
L.~Ljung.
\newblock {\em System identification}.
\newblock Prentice Hall, 1987.

\bibitem{Mayo2007634}
A.~Mayo and A.~Antoulas.
\newblock A framework for the solution of the generalized realization problem.
\newblock {\em Linear Algebra and its Applications}, 425(2-3):634 -- 662, 2007.

\bibitem{Mendel}
J.~Mendel.
\newblock Tutorial on higher-order statistics (spectra) in signal processing
  and system theory: theoretical results and some applications.
\newblock {\em Proc. IEEE}, 79:278--305, 1991.

\bibitem{PReProj}
B.~Peherstorfer.
\newblock Sampling low-dimensional {Markovian} dynamics for preasymptotically
  recovering reduced models from data with operator inference.
\newblock {\em SIAM Journal on Scientific Computing}, 42(5):A3489--A3515, 2020.

\bibitem{PSW16TLoewner}
B.~Peherstorfer, S.~Gugercin, and K.~Willcox.
\newblock Data-driven reduced model construction with time-domain {Loewner}
  models.
\newblock {\em SIAM Journal on Scientific Computing}, 39(5):A2152--A2178, 2017.

\bibitem{pehersto15dynamic}
B.~Peherstorfer and K.~Willcox.
\newblock Dynamic data-driven reduced-order models.
\newblock {\em Computer Methods in Applied Mechanics and Engineering},
  291:21--41, 2015.

\bibitem{Peherstorfer16DataDriven}
B.~Peherstorfer and K.~Willcox.
\newblock Data-driven operator inference for nonintrusive projection-based
  model reduction.
\newblock {\em Computer Methods in Applied Mechanics and Engineering},
  306:196--215, 2016.

\bibitem{Peherstorfer16AdaptROM}
B.~Peherstorfer and K.~Willcox.
\newblock Dynamic data-driven model reduction: Adapting reduced models from
  incomplete data.
\newblock {\em Advanced Modeling and Simulation in Engineering Sciences},
  3(11), 2016.

\bibitem{PWG17MultiSurvey}
B.~Peherstorfer, K.~Willcox, and M.~Gunzburger.
\newblock Survey of multifidelity methods in uncertainty propagation,
  inference, and optimization.
\newblock {\em SIAM Review}, 60(3):550--591, 2018.

\bibitem{Penzl2006322}
T.~Penzl.
\newblock Algorithms for model reduction of large dynamical systems.
\newblock {\em Linear Algebra and its Applications}, 415(2-3):322 -- 343, 2006.

\bibitem{doi:10.2514/6.2019-3707}
E.~Qian, B.~Kramer, A.~N. Marques, and K.~E. Willcox.
\newblock Transform \& learn: A data-driven approach to nonlinear model
  reduction.
\newblock In {\em AIAA Aviation 2019 Forum}, 2019.

\bibitem{Qin20061502}
S.~J. Qin.
\newblock An overview of subspace identification.
\newblock {\em Computers \& Chemical Engineering}, 30(10-12):1502 -- 1513,
  2006.

\bibitem{FIRSystem}
L.~Rabiner, R.~Crochiere, and J.~Allen.
\newblock {FIR} system modeling and identification in the presence of noise and
  with band-limited inputs.
\newblock {\em IEEE Transactions on Acoustics, Speech, and Signal Processing},
  26(4):319--333, 1978.

\bibitem{Reynders2012}
E.~Reynders.
\newblock System identification methods for (operational) modal analysis:
  Review and comparison.
\newblock {\em Archives of Computational Methods in Engineering},
  19(1):51--124, 2012.

\bibitem{FLM:6837872}
C.~Rowley, I.~Mezi\'{c}, S.~Bagheri, P.~Schlatter, and D.~Henningson.
\newblock Spectral analysis of nonlinear flows.
\newblock {\em Journal of Fluid Mechanics}, 641:115--127, 12 2009.

\bibitem{Rudye1602614}
S.~H. Rudy, S.~L. Brunton, J.~L. Proctor, and J.~N. Kutz.
\newblock Data-driven discovery of partial differential equations.
\newblock {\em Science Advances}, 3(4), 2017.

\bibitem{Schaeffer6634}
H.~Schaeffer, R.~Caflisch, C.~D. Hauck, and S.~Osher.
\newblock Sparse dynamics for partial differential equations.
\newblock {\em Proceedings of the National Academy of Sciences},
  110(17):6634--6639, 2013.

\bibitem{doi:10.1137/18M116798X}
H.~Schaeffer, G.~Tran, and R.~Ward.
\newblock Extracting sparse high-dimensional dynamics from limited data.
\newblock {\em SIAM Journal on Applied Mathematics}, 78(6):3279--3295, 2018.

\bibitem{FLM:7843190}
P.~Schmid.
\newblock Dynamic mode decomposition of numerical and experimental data.
\newblock {\em Journal of Fluid Mechanics}, 656:5--28, 8 2010.

\bibitem{SchmidDMD}
P.~Schmid and J.~Sesterhenn.
\newblock Dynamic mode decomposition of numerical and experimental data.
\newblock In {\em Bull. Amer. Phys. Soc., 61st APS meeting}, page 208. American
  Physical Society, 2008.

\bibitem{SCHULZE2016125}
P.~Schulze and B.~Unger.
\newblock Data-driven interpolation of dynamical systems with delay.
\newblock {\em Systems \& Control Letters}, 97:125 -- 131, 2016.

\bibitem{SCHULZE2018250}
P.~Schulze, B.~Unger, C.~Beattie, and S.~Gugercin.
\newblock Data-driven structured realization.
\newblock {\em Linear Algebra and its Applications}, 537:250 -- 286, 2018.

\bibitem{2019arXiv190803620S}
R.~Swischuk, B.~Kramer, C.~Huang, and K.~Willcox.
\newblock Learning physics-based reduced-order models for a single-injector
  combustion process.
\newblock {\em AIAA Journal}, 58(6):2658--2672, 2020.

\bibitem{Tu2014391}
J.~H. Tu, C.~W. Rowley, D.~M. Luchtenburg, S.~L. Brunton, and J.~N. Kutz.
\newblock On dynamic mode decomposition: Theory and applications.
\newblock {\em Journal of Computational Dynamics}, 1(2):391--421, 2014.

\bibitem{Viberg19951835}
M.~Viberg.
\newblock Subspace-based methods for the identification of linear
  time-invariant systems.
\newblock {\em Automatica}, 31(12):1835 -- 1851, 1995.
\newblock Trends in System Identification.

\bibitem{wainwright_2019}
M.~J. Wainwright.
\newblock {\em High-Dimensional Statistics: A Non-Asymptotic Viewpoint}.
\newblock Cambridge Series in Statistical and Probabilistic Mathematics.
  Cambridge University Press, 2019.

\end{thebibliography}
\bibliographystyle{abbrv}

\end{document}